\documentclass[oneside, 12pt]{amsart}

\usepackage{amsmath, amssymb, amsthm, hyperref, graphicx,   verbatim, enumitem, xcolor, tikz-cd, stmaryrd, epsfig, xypic}
\usepackage{array, colortbl}
\usepackage{mathrsfs}  
\usepackage{mathabx}
\usepackage{times}
 
 \DeclareFontFamily{U}{mathc}{}
\DeclareFontShape{U}{mathc}{m}{it}%
{<->s*[1.03] mathc10}{}

\DeclareMathAlphabet{\mathscr}{U}{mathc}{m}{it}

\usepackage[paper=a4paper, marginpar=2.4cm]{geometry}
\usepackage[all]{xy}
\usepackage[utf8]{inputenc}

\usepackage[english]{babel}
 
\newcommand{\e}{\varepsilon}

\newcommand{\fr}{\partial}

\newcommand{\set}[1]{\left\{#1\right\}}
\newcommand{\norm}[1]{{\left\Vert#1\right\Vert}}
\newcommand{\abs}[1]{\left\vert#1\right\vert}

\newcommand{\rest}[1]{ \arrowvert_{#1}}

\newcommand{\el}{\mathcal{L}}

\newcommand{\lrpar}[1]{\left(#1\right)}
\newcommand{\bra}[1]{\left\langle #1\right\rangle}

\newcommand{\loc}{\mathrm{loc}}
\newcommand{\inv}{^{-1}}
\renewcommand{\d}{\mathrm{d}}

\DeclareMathOperator{\supp}{Supp}

\DeclareMathOperator{\jac}{Jac}

\DeclareMathOperator{\Stab}{Stab}

\newcommand{\C}{\mathbf{C}}
\newcommand{\R}{\mathbf{R}}

\newcommand{\Z}{\mathbf{Z}}
\newcommand{\N}{\mathbf{N}}

\renewcommand{\P}{\mathbb{P}}

\newcommand{\Aut}{\mathsf{Aut}}

\newcommand{\GL}{{\sf{GL}}}

\theoremstyle{plain}

\newtheorem{thm}{Theorem}[section]
\newtheorem{cor}[thm]{Corollary}

\newtheorem{lem}[thm]{Lemma}

\newtheorem{que}[thm]{Question} 
 
\newtheorem{conj}[thm]{Conjecture} 

\newtheorem*{thm*}{Theorem}

\theoremstyle{definition}

\newtheorem{rem}[thm]{Remark}

\numberwithin{equation}{section}       % Number formulas within sections

\addtocounter{section}{0}             % Start with section 1
\numberwithin{equation}{section}       % Number formulas within sections
\title[same equilibrium  measure]{Polynomial automorphisms of $\C^2$ with the same equilibrium measure}

\date{\today}

 \author{Romain Dujardin}
\address{Romain Dujardin,  Sorbonne Universit\'e, CNRS, Laboratoire de Probabilit\'es, Statistique  et Mod\'elisation  (LPSM), F-75005 Paris, France}
\email{romain.dujardin@sorbonne-universite.fr}
\thanks{{\footnotesize{The research activities of the author are partially supported by the  French National Research Agency under the project DynATrois (ANR-24-CE40-1163). }}}

\begin{document}

\setlength{\parskip}{.2em}
\setlength{\baselineskip}{1.26em}  

%%%%%%%%%%%%%%%%%%%%%%%%%%%%%%%%%%%%%%%%%%%%%%%%%%%%% 

\begin{abstract}
We review the following question: when do two polynomial automorphisms of $\C^2$ have the same 
maximal entropy measure? 
While the answer has been known for some time when both automorphisms 
are compositions of Hénon maps, the general case is more delicate. 
In this (mostly expository) note we explain how recent results in random dynamics give 
 an almost complete solution to the question, and point out a number of  related open problems. 
\end{abstract}

 \maketitle
   
\section*{Introduction}

\subsection*{One-dimensional background} 
A  venerable  problem in holomorphic dynamics   is to classify pairs of holomorphic maps in a certain class sharing the same Julia set. This  goes back to the work of Fatou, Julia  and Ritt 
on the functional equation $f\circ g = g\circ f$ for  rational maps in one variable ~\cite{fatou1, fatou2, julia, ritt}. 
The interest in this question has   
 been renewed by considerations from  arithmetic dynamics, 
and there is now a considerable  body of literature on this problem  and its variants (see \cite{zhang} for
a recent review and  further references).
Let us briefly discuss  a few milestones. If $f$ is a rational map, we denote by $J_f$ its Julia set and 
$\mu_f$ its equilibrium measure, that is its unique probability measure of maximal entropy. 
% let $f$ and $g$ be rational maps on $\P^1(\C)$. Let us 
% say that $f$ and $g$ are algebraically related if $(f,g)$ admits an invariant curve in $\P^1\times \P^1$. 
Let $f$ and $g$ be non-integrable rational maps.  
\begin{enumerate}
\item The work of Fatou~\cite{fatou1} and Julia~\cite{julia} implies
that if $J_f=J_g$ is not a curve or the whole sphere  and $f$ and $g$ share a repelling periodic point, 
then $f$ and $g$ have a common iterate.
\item Levin and Przytycki~\cite{levin-przytycki} used an ergodic theoretic argument together with the work of Levin on local symmetries of Julia sets~\cite{levin} to show that if $\mu_f  = \mu_g$ then $f$ and $g$ are algebraically related. 
\item It follows that  if $f$ and $g$ are polynomials with $J_f = J_g$ then $f$ and $g$ satisfy a relation of the form $f^N= L\circ g^M$, where $L$ is a linear map of finite order, a theorem originally due to Beardon \cite{beardon}.
\end{enumerate}
 An important fact is that for polynomials, the  Julia set
determines the equilibrium measure, hence the equality of Julia sets  implies  the equality of the Green functions, in which case tools such as the Böttcher function can be used to conclude. By contrast, in the rational case this approach is not applicable, and more sophisticated dynamical arguments are required. 
Note that in full generality, the problem of determining pairs of rational maps with the same 
non-smooth Julia set is still open.
%Variants of this problem have been considered, such as the possibility of local biholomorphisms between parts of the Julia set~\cite{dfg, ji-xie:symmetries}. 
 
 \subsection*{Two-dimensional results} 
 In this note we address a similar question  for  polynomial automorphisms of $\C^2$.
Recall from Friedland-Milnor~\cite{friedland-milnor} that any polynomial automorphism of $\C^2$ with  positive topological entropy  is conjugate to a composition of Hénon maps. 
A first important   observation is that when $f$ and $g$ are both composition of Hénon maps (or more generally regular automorphisms), or more generally if they are simultaneously conjugate to maps in such a form, then the $J_f=J_g$ problem is completely solved (see \S~\ref{subs:loxo} below for the notion of Julia set in this context, and other basic facts) . 

\begin{thm*}[Lamy~\cite{lamy:these, lamy}, see also Dujardin-Favre~\cite{dujardin-favre:DMM}]
Let $f$ and $g$ be compositions of Hénon maps in $\C^2$, and assume that  $J_f=J_g$ (resp. $K_f = K_g$). Then  there exists $M, N\in \Z^*$ such that $f^N=g^M$. 
\end{thm*}

Indeed Lamy showed in his thesis\footnote{There is also a general statement about  $\mu_f=\mu_g$  in~\cite{lamy:these}. Note however that the definition of $\mu_f$ is different there so that it does not imply  the results of this article}
that if $f, g \in \Aut(\C^2)$ are such that 
 $G^+_f = G^+_g$ then  there exists $M, N\in \N^*$ such that $f^N=g^M$. 
If furthermore $f$ and $g$ are both compositions of Hénon maps (or more generally regular automorphisms),  then  $J_f=J_g$ (or $K_f = K_g$) implies that $\mu_f  = \mu_g$. Indeed, $\mu_f$  is the pluricomplex harmonic measure of the Julia set, which depends on the choice of a ``polynomial model of $\C^2$'', that is, of a divisor at infinity in a compactification of $\C^2$. In such a setting, it is not hard to deduce that  $\mu_f  = \mu_g$ implies  $G^+_f = G^+_g$, hence  $f^N=g^M$ (see the proof of Theorem~C in~\cite{dujardin-favre:DMM}).  
A  notable  result from~\cite{dujardin-favre:DMM} is that if for every $h\in \bra{f,g}$ we have 
 $K_h = K_f$ (or $J_h = J_f$), then it can be concluded that $f^N=g^M$.

Therefore,  the focus in this note is on the situation where   $f$ and $g$ are \emph{not} 
simultaneously conjugate to regular automorphisms. In this case  both 
 implications $J_f=J_g \Rightarrow \mu_f  = \mu_g$ and $\mu_f  = \mu_g \Rightarrow T^+_f = T^+_g$ are unclear (see~\S~\ref{subs:loxo} for a   detailed discussion). In this respect, the situation is comparable to that of rational maps in one variable (as opposed to polynomials), and
  it should be expected that more sophisticated ergodic tools are required to solve the problem. 

In this paper we explain how the ergodic theory of random compositions of holomorphic
diffeomorphisms developed by Cantat and the  author, as well as Roda and Nerrière~\cite{stiffness, roda, nerriere} 
solves the $\mu_f  = \mu_g$ problem in the vast majority of cases (see Theorems~\ref{thm:same_eq_measure_dissipative} and~\ref{thm:same_eq_measure_conservative}). 
The mathematical work was actually already done in these references, and this contribution 
should mostly be viewed as an opportunity to showcase these techniques, gather otherwise scattered arguments, and advertise   some open problems. Interestingly, the core technical result underlying these measure rigidity results in random dynamics (based on the work of Brown and Rodriguez-Hertz~\cite{brh})
is   a kind of distant cousin of Levin's paper on symmetries of the Julia set (in which an invariance under a flow is obtained from an accumulation of discrete symmetries).

On the other hand, in this generality the  $J_f =J_g$ problem is widely open, 
even if some cases can be settled, notably in the hyperbolic case. We plan to come back to this issue in a forthcoming  work.
Observe however that $J_f =J_g$
still implies $\mu_f \asymp \mu_g$ (as follows from Levenberg~\cite{levenberg}, see also  Theorem~\ref{thm:GZ}) 
so this problem still might  be approachable  by 
ergodic theoretic methods (see Remark~\ref{rem:Jf=Jg}).

Rigidity questions for pairs of polynomial automorphisms of $\C^2$ 
have attracted considerable attention in recent years. 
The rigidity of conjugacies was explored by Cantat and the author in~\cite{conjugate} and~\cite{rigidity} and multiplier rigidity in~\cite{multiplier};  
Bera, Pal and Verma have 
studied automorphisms with biholomorphic   escaping sets in 
in~\cite{bera-pal-verma,pal,bera-verma}; and Vigny recently solved the $J_f^+  = J_g^+$ problem for regular automorphisms in higher dimension~\cite{vigny}. 
In a slightly different setting,  Zhang~\cite{zhang} obtained
 measure rigidity results for regular polynomial maps of $\C^2$. 

\subsection*{Contents} After some preliminaries on loxodromic automorphisms, we develop the  
 main rigidity  results in Sections~\ref{sec:non-conservative} (non-conservative case) and~\ref{sec:conservative} (conservative case). Further results on the classification of stationary measures are presented in Section~\ref{sec:further}. 
 In the Appendix, we discuss the action of polynomial automorphisms on the pluricomplex harmonic measure. 

\subsection*{Acknowledgments} 
I am grateful to Serge Cantat for our long-term collaboration on random dynamics 
and many discussions on the topics of this paper. I extend my  
thanks to  Sébastien Boucksom, Vincent Guedj and Ahmed Zeriahi for their input in pluripotential theory, and to Yugang Zhang for some insightful comments. 
The main body of the paper was not touched  by AI at all. 
On the other hand the calculations of Lemma~\ref{lem:density_at_0} in the Appendix
were carried out by (the free version of) ChatGPT.  

\section{Preliminaries}
\subsection{Loxodromic automorphisms}\label{subs:loxo}
A polynomial automorphism of $f$ of $\C^2$ 
 is said loxodromic if its dynamical degree $d  := \lim_{n\to \infty}\deg(f^n)^{1/n}$ is greater than 1, or equivalently its topological degree is positive. 
We briefly recall a few standard concepts and facts, mostly to
 fix notation (see~\cite{bs1, bs2, bls, bs6} for more details and  references). Assume first that 
 $f$ is a composition of Hénon maps, or more generally that it is regular in the sense of Sibony, that is 
 $I(f)\cap I(f\inv) = \emptyset$, where $I(f)$ is the indeterminacy locus of the extension of $f$ to a birational map on $\P^2(\C)$. Then: 
\begin{itemize}
\item $K_f^+$ is the set of points with bounded forward orbits, and $J_f^+ = \fr K_f^+$ is the forward Julia set, and similarly for $K_f^-$, $J_f^-$;
\item $J_f = J_f^+\cap J_f^-$ is the Julia set;
\item $G_f^{+/-}(z)  = \lim_{n\to \infty} d^{-n} \log^+\norm{f^{+/- n}(z)}$ are the dynamical Green functions and satisfy  $G_f^\pm\circ f^{\pm1}  = dG_f^\pm$;
\item $T_f^\pm =  \d\d^c G_f^\pm$ are the dynamical Green currents;
\item $\supp T_f^\pm = J_f^\pm$ and $T_f^\pm$ is the only closed positive current supported on $K_f^\pm$; in particular if $p$ is any saddle point, $W^{s}(p)$ is an entire curve whose    only 
associated Ahlfors current   is $T_f^{+}$ and likewise for $T_f^-$, $J_f^-$ and  $W^u(p)$; 
\item $\mu_f = T_f^+\wedge T_f^-$ is the unique measure of maximal entropy; its support is 
contained in $J_f$ and denoted $J^*_f$;  it is ergodic, hyperbolic 
and its  Lyapunov exponents satisfy $\chi^-(\mu_f)\leq -\log d < \log d \leq \chi^-(\mu_f)$;  
\item if $\abs{\jac(f)}\leq 1$, $J_f$ is connected iff $J^*_f$ is connected iff $\chi^+(\mu_f) = \log d$;
\item $G_f = \max(G_f^+, G_f^-)$ is the   extremal (or Green) function of 
$K_f$ (or $J_f$) from pluripotential theory 
and $\mu_f = (\d \d^cG_f)^2$ is  the pluricomplex harmonic measure of  $K_f$.
\end{itemize}
An important point in this paper is that  by design 
we do not work with the Friedland-Milnor normalization. 
This induces a few changes with respect to the usual  ideas. 
%By~\cite{bs1}, if $f$ is a composition of Hénon maps (or more generally a regular automorphism) then the maximal entropy measure $\mu_f$ is the pluricomplex harmonic measure of $J_f$ or $K_f$, and $G_f = \max(G_f^+, G_f^-)$ is the associated   extremal (or Green) function.  
Take a regular automorphism $h$ and conjugate it 
in $\mathrm{Aut}(\C^2)$ to get $f= g  h  g\inv$. Assume that $g$ is not linear. 
We still define $\mu_f$ to be unique (probability) measure of maximal entropy for $f$, but now its relationship with the invariant currents is slightly different. 
Define $T_f^+$ to be the unique positive closed current of mass 1 in $\P^2$ such that $f^*T_f^+ =d T_f^+$. Then $T_f^+$ is the trivial extension to $\P^2$ of  $c (g_*T_h^+)\rest{\C^2}$ for some $c>0$. 
We also set 
$G^\pm_f = c G^+_h\circ g\inv$, and  the same for $T_f^-$ and $G_f^-$. The wedge product  
$T_f^+ \wedge T_f^-$ is a probability measure that may put some mass at infinity, so 
 $\mu_f  = c' T_f^+ \wedge T_f^- \rest{\C^2}$ for some $c'>0$. 
Also, $\tilde G_f:= \sqrt{c'} \max (G_f^+, G_f^-)$ is a psh function such that 
$(\d\d^c \tilde G_f) ^2 = \mu_f = g_*\mu_h$ but it does not grow like $\log\norm{z}$ at infinity
so it
is \emph{not} the pluricomplex Green function
of $J_f$, neither a ``normalized'' potential of $\mu_f$. 
This indicates that we cannot easily recover $G^\pm_f$ from 
$\mu_f$. This 
also suggests that $\mu_f$ may not be the harmonic mesure of $J_f$, 
so that $J_f$ does not directly determines $\mu_f$. 
See the Appendix for explicit (non-dynamical) examples where this phenomenon happens. 

The structure of the group $\Aut(\C^2)$ of polynomial automorphisms of $\C^2$ is well-understood
(see Lamy~\cite{lamy} and also~\cite{lamy:book} for a comprehensive treatment). Its amalgamated 
product structure gives rise to a faithful action of $\Aut(\C^2)$ by isometries 
on a simplicial tree (Bass-Serre), and the action of a loxodromic element preserves 
geodesic $\mathrm{Geo}(f)$ (its axis) on which it acts as a non-trivial translation. A subgroup of $\Aut(\C^2)$ is said \emph{non-elementary} if it contains two loxodromic elements $f,g$ 
 ``in general position'' that is, whose axes have disjoint endpoints.  
 %In particular, $\bra{f^N, g^N}$ is isomorphic to a free group on two generators for some $N$. 

\subsection{Random dynamics and $\GL_2(\C)$-cocycles}\label{subs:random_GL2}
 (See~\cite[\S 7]{stiffness} for a more detailed presentation in the similar context  of automorphisms of complex surfaces, and~\cite{dujardin:book} for generalities on the ergodic theory of 
 random dynamical systems.)
If $\nu$ is a probability measure on $\Aut(\C^2)$, we can define a  
random dynamical system on $\C^2$ by 
compositing independents   elements with distribution $\nu$. 
 A probability measure $\mu$ on $\C^2$ is $\nu$-invariant 
if $h_*\mu = \mu$ for $\nu$-a.e. $h$   and it is
$\nu$-stationary if $\int h_*\mu \d \nu(h)  =\mu$. 
In this paper we will only consider the case where $\nu$ is finitely supported and $\mu$ has compact support, and all measures considered are probability measures.  The interesting situation is when 
$\bra{\supp(\nu)}$ is non-elementary
A $\nu$-stationary measure $\mu$  is said ergodic if every measurable $\nu$-a.s.-invariant subset 
is null or co-null; note that in our finitely generated context,  
$\nu$-a.s.-invariant simply means $\bra{\supp(\nu)}$-invariant. 
Every invariant measure is stationary, and the stiffness problem is whether the converse holds. This problem now admits   
 an fairly  complete solution, as we explain below.
 %a combination of the ideas of~\cite{brh, stiffness, roda, nerriere}. 

Such a random dynamical system, endowed with a stationary measure $\mu$, gives rise to a skew-product and a  $\GL_2(\C)$-cocycle as follows. 
Let $\Sigma_ +  =\Aut(\C^2)^\N$, endowed with the one-sided shift $\sigma$ and the shift invariant measure $\nu^\N$. Then,    we define a   skew-product map  $F_+$
 on $\Sigma_+\times \C^2$   by $F_+(\omega, x)  =(\sigma(\omega), f_{\omega_0}(x))$. It admits 
  $\nu^\N\times \mu$ as an invariant measure, which is ergodic if and only if $\mu$ is ergodic as a stationary measure. We can extend $F_+$ as an invertible skew-product $F$
   over the 2-sided shift $(\Sigma, \sigma, \nu^\Z)$, 
   and there is a unique $F$-invariant measure $\mathscr{m}$ projecting down to 
$\nu^\N\times \mu$. A fundamental result is that $\mathscr{m} = \nu^\Z\times \mu$ if and only 
if $\mu$ is $\nu$-invariant. 
 
 The tangent action of the random dynamics yields  a $\mathrm{GL}_2(\C)$-cocycle over 
 $F_+$ (resp. $F$)  defined by 
 $(\omega, x)\mapsto Df_{\omega_0}(x)$. 
 The Lyapunov exponents $\lambda_1\geq \lambda_2$ 
 of $\mu$ viewed as a $\nu$-stationary measure are by definition the Lyapunov exponents 
 $\lambda_1\leq \lambda_2$ of this cocycle. 
 An important observation is that $\lambda_1 + \lambda_2$, hence in particular its sign, 
 can  easily be computed (see~\cite{simons})
\begin{equation}\label{eq:lyapunov}
\lambda_1+\lambda_2 = \int \log\abs{\jac(h)}d\nu(h). 
\end{equation}
 
Coming back to the stiffness problem for $\Aut(\C^2)$,   
 the Invariance Principle of Crauel~\cite{crauel} and Avila-Viana~\cite{avila-viana} (see~\cite{ledrappier:survey} for a recent account on this topic) 
 asserts that $\lambda_1\geq 0$ implies  invariance. 
Thus, when $\lambda_1+\lambda_2\geq 0$, 
 we only have to deal with the case where 
$\lambda_1<0<\lambda_2$, that is, when $\mu$ is hyperbolic as a stationary measure. 
This is the situation considered in~\cite{brh, stiffness, roda} and in particular \cite{nerriere}, and the current status of the problem is that if $\bra{\supp(\nu)}$ is purely loxodromic, then $\mu$ must be invariant and satisfy further rigidity properties. It is likely that the result holds under the sole assumption that $\bra{\supp(\nu)}$ is non-elementary. 

\begin{rem}\label{rmk:purelox} 
Lamy's classification~\cite[\S~4]{lamy} says that $\bra{f,g}$ is non-elementary 
if and only if the corresponding geodesics 
$\mathrm{Geo}(f)$ and $\mathrm{Geo}(g)$ in the Bass-Serre tree  do not coincide. Then there exists an integer $N$ such that $\bra{f^N, g^N}$ is purely loxodromic and 
isomorphic to a free group on two generators. Thus,   questions such as~$\mu_f = \mu_g$ or 
$J_f = J_g$   can 
be reduced from non-elementary to the purely loxodromic case. 
\end{rem}

\section{Non-conservative maps}\label{sec:non-conservative}

\begin{thm}\label{thm:same_eq_measure_dissipative}
Let $f$ and $g$ be  polynomial automorphisms of $\C^2$ of positive entropy 
such that $\mu_f = \mu_g$. Assume that one of $f$ or $g$ is not conservative. Then there exists $(M, N)\in (\Z^*)^2$ such that 
$f^N = g^M$. 
\end{thm}

This result is a consequence of the following   more general statement. 

\begin{thm}\label{thm:same_positive_entropy_measure}
Let $f$  be a non-conservative  polynomial automorphism of $\C^2$  and 
$\mu$ be  a measure of positive entropy for $f$. If $g$ is another automorphism such that 
$\mu$ is  $g$-invariant,  there exists $(M, N)\in (\Z^*)^2$ such that 
$f^N = g^M$. 
\end{thm}

This theorem  is a combination  of~\cite[\S 11]{stiffness}, which is adapted in~\cite{simons} for     real  polynomial automorphisms,  together with the main classification theorem of Roda~\cite{roda}, as well as 
 an important input from Brown and Rodriguez-Hertz~\cite[Thm 5.1]{brh}. 
The case where $\mu$ is not ergodic is discussed in~\cite[Thm 8.8]{hyperbolic}. 
We give a reasonably complete account of the proof   to help the reader navigate these references.
Note that  in this positive entropy setting, the work of Nerrière~\cite{nerriere} is not required. 

\begin{rem}
It is interesting that there is no known 
analogue of Theorem~\ref{thm:same_positive_entropy_measure} for 
one-dimensional rational maps. To fully appreciate the difficulty of this kind of 
problem, it is worth noting that 
classifying measures simultaneously 
invariant under $z\mapsto z^2$ and $z\mapsto z^3$ is  the celebrated  
Furstenberg $\times 2\times 3$ problem (even if this example is probably  misleading since it concerns an abelian semi-group, and in addition its  positive entropy case is actually understood by~\cite{rudolph}). 
\end{rem}

\begin{proof}[Proof of Theorem~\ref{thm:same_positive_entropy_measure}]
We will show that the group $\Stab(\mu)$ of automorphisms of $\C^2$ preserving $\mu$ is elementary.  More precisely, since $\mu$ has positive entropy, $f$ must be loxodromic, 
and we will show that  
every element $h\in \Stab(\mu)$ preserves the invariant geodesic $\mathrm{Geo}(f)$ 
 of $f$ in the Bass-Serre tree of $\Aut(\C^2)$. Then it follows from 
 from the classification of subgroups of $\Aut(\C^2)$ (see in particular~\cite[Prop. 4.10]{lamy})
  that $\Stab(\mu)$ is virtually cyclic  and the result follows. 
 Let $g  = hfh\inv$ so that $g\in\Stab(\mu)$ is loxodromic with $\mathrm{Geo}(g) = h_*(\mathrm{Geo}(f))$
  and let us show that $\mathrm{Geo}(g)  = \mathrm{Geo}(f)$. 

Assume first that $\mu$ is ergodic under $f$. Fix a measure $\nu$ on $\bra{f,g}$ of the form $\nu = a\delta_f+b\delta_{f\inv}+c\delta_g + d \delta_{g\inv}$, where $a$, $b$, $c$, $d$ are positive real numbers summing to 1.
 Put $\alpha =  (a,b,c,d)$ and $\nu_\alpha$ the corresponding measure. 
Then $\mu$ is ergodic as a $\nu$-stationary measure: indeed any $\bra{f,g}$-invariant set must be $f$-invariant. The argument will be decomposed into several cases, according to the existence or non-existence of certain invariant complex line fields. Denote $\Gamma:=\bra{f,g}$.

\noindent \textit{Case 1. There is a measurable $\Gamma$-invariant field of complex lines $x\mapsto E(x)\subset T_x\C^2$.}

 Since $\mu$ has positive entropy for $f$, by the Ruelle inequality it is hyperbolic, and by ergodicity we conclude that $E(x)$ must coincide with one of the Oseledets directions: $E(x) = E^u_f(x)$ or $E(x) = E^s_f(x)$, $\mu$-a.e., and likewise for $g$. Thus, replacing $f$ by $f\inv$ (resp. $g$ by $g\inv$) if necessary, we get that $E(x)   = E^s_f(x) = E^s_g(x)$ a.e. In other words, the Pesin stable manifolds of $f$ and $g$ are tangent a.e. In this case,  we will show that these stable manifolds actually coincide. 
 Since by~\cite{bs2} every stable manifold is dense in $J^+$, we conclude that 
 $J^+_f = J^+_g$ which by uniqueness implies $T^+_f = T^+_g$~\cite{fornaess-sibony:1993}, 
hence by~\cite[Thm 5.4]{lamy}
 $\mathrm{Geo}(g)  = \mathrm{Geo}(f)$, as asserted. 
 
 In a nutshell, the argument runs as follows (see~\cite[Lem 11.3]{stiffness}) for technical  details): assume by contradiction that stable manifolds do not coincide on a set of positive measure, hence almost everywhere by ergodicity. Consider a 
measurable partition $\eta^s_f$ made of open subset of local stable manifolds of $f$ (a so-called Pesin partition). The positivity of entropy implies that the conditionals of $\mu$ along the atoms of this partition are non-atomic. So in a typical atom of $\eta^s_f$, which is an open set in a  local stable manifold $W^s_\loc(x)$, we have uncountably many ``typical'' points $y$ for $g$, accumulating $x$. 
These stable manifolds have size bounded from below, they are all tangent to $W^s_\loc(x)$ and they must be disjoint:  an elementary complex analysis lemma~\cite[Lem. 6.4]{bls} shows that this is impossible (see Figure~\ref{fig:tangent}).  

\begin{figure}[h]
\includegraphics[width=0.5\textwidth]{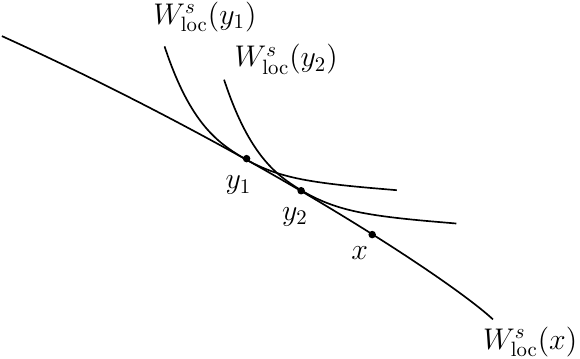}
\caption{Tangency is incompatible with disjointness.}\label{fig:tangent}
\end{figure}

\noindent \textit{Case 2. There is a measurable pair of  $\Gamma$-invariant line fields.} 
 
 Here we have a pair of line fields $\set{E_1(x), E_2(x)}$ which is $\Gamma$-invariant but a priori not individually invariant by $f$ and $g$. In any case by Oseledets' Theorem 
  at $\mu$-a.e. $x$ we have 
 $$\set{E_1(x), E_2(x)} = \set{E^s_f(x), E^u_f(x)} = \set{E^s_g(x), E^u_g(x)}. $$
 Replacing $g$ by $g\inv$ if necessary, we have $E^s_f(x) = E^s_g(x)$ on a set of positive measure, and we conclude exactly as in the previous case. 
 
\begin{rem}\label{rmk:conservative1}
Up to this point we have not used the fact that $f$ is non-conservative. 
\end{rem}

 \noindent \textit{Case 3.  None of cases 1 and 2 holds. }
 
 This is the most sophisticated case to deal with. Here we look for a contradiction. For this, 
  we want to find  a parameter
 $\alpha = (a,b,c,d)$ as above so that $\mu$ is hyperbolic as a $\nu_\alpha$-stationary measure. 
  This  closely follows the proof  of \cite[Thm 5.1]{brh}.
 Then we can apply the measure classification theorem of~\cite{roda} to prove that this case leads to a contradiction.

 \noindent \textit{Step 3.1. Choice of $(a,b,c,d)$.} 
 
Consider the simplex $\Delta = \set{\alpha = (a,b,c,d)\in (\R_{>0})^4, a+b+c+d =1}$. Since $\mu$ is ergodic as a $\nu_\alpha$-stationary measure, it admits two Lyapunov exponents 
$\lambda^1_\alpha\leq \lambda^2_\alpha$. Since $f$ and $g$ have constant Jacobian, as in the deterministic case we can easily compute the sum of the exponents (see~\cite[Prop. 10.2]{simons})
\begin{equation}\label{eq:sum_exponents}
\lambda^1_\alpha+ \lambda^2_\alpha = (a-b)\log\abs{\jac f} + (c-d) \log\abs{\jac g}. 
\end{equation}

We have a fairly good understanding  of the cases where $\lambda^1_\alpha= \lambda^2_\alpha$. 
For this, consider the projectivized tangent bundle $\P T\C^2$, endowed with the 
dynamical system induced by $\Gamma$ (resp. the random dynamical system induced by $\nu_\alpha$). 
Consider the $\nu_\alpha$-stationary 
probability  measures $\hat\mu$ on  $\P T\C^2$ 
whose projection on $\C^2$ coincides with $\mu$, 
 i.e.\ 
$\pi_\varstar \hat\mu= \mu $ where $\pi\colon \P T\C^2\to \C^2$ is the natural projection. . By compactness of $\supp(\mu)$ and of the fibers, such measures always exist. Any such measure can be disintegrated along the fibers of  $\pi$ and the conditional $\hat \mu(\cdot\vert \pi\inv(x)$ is denoted 
$\hat\mu_x$. 
Ledrappier's Invariance Principle 
(see~\cite{ledrappier:stationary} and~\cite{ledrappier:survey} for a recent overview) 
asserts that 
if $\lambda^{1}_\alpha = \lambda^{2}_\alpha$, 
then for any  stationary measure $\hat \mu$ on $\P TX$ projecting   to $\mu$, we have $\P(D_xh)_\varstar \hat\mu_x = \hat\mu_{h(x)}$
for $\mu$-a.e. $x$ and every $h\in\Gamma$.
 
 The above  invariance relation should be considered as an exceptional property, and the situations where this happens can be classified (see~\cite[Thm 7.3]{hyperbolic} for this statement, which is a slight generalization of~\cite{barrientos-malicet}): either there exist an invariant or a pair of invariant line fields, or the projective action of $\Gamma$ is reducible to a compact subgroup (i.e. it acts like a compact subgroup up to a change of coordinates). None of these cases can 
 happen in our case: the existence of line fields is excluded by definition, and the reduction to a compact subgroup is impossible because $\mu$ is hyperbolic for $f$. The conclusion is that 
 $\lambda^1_\alpha< \lambda^2_\alpha$ for every $\alpha\in \Delta$.

The last step is a continuity result for the Lyapunov exponents. Indeed the assumption of Case 3 and the inequality  $\lambda^1_\alpha< \lambda^2_\alpha$ imply that $\lambda^1_\alpha$ and $\lambda^2_\alpha$ are continuous (this follows from the techniques of Viana~\cite{viana:book, avila-viana}, see \cite[Prop. 13.7]{brh} for details\footnote{One difference between this reference and our setting   in that the dynamics considered in~\cite{brh} is volume preserving.  So the proof in~\cite{brh} has to be modified  so that the role of 0 in their argument is now played by the average exponent
$(\lambda^1_\alpha+ \lambda^2_\alpha)/2$. Algebraically, this corresponds to factorizing the determinant and work with the associated projective cocycle. }). 
Since $\mu$ is $\Gamma$-invariant, it follows from a classical argument  in Pesin's theory  that 
$\lambda^1_\alpha$ and $\lambda^2_\alpha$ cannot be both positive or both  negative, for this would force $\mu$ to be atomic. The idea is as follows: with notation as in \S~\ref{subs:random_GL2}, if 
 $\lambda^1_\alpha$ and $\lambda^2_\alpha$ are both  negative, say, 
 ergodicity and contraction (see e.g.~\cite[Cor. S.5.2]{Katok-Hasselblatt}) 
  imply that the conditional measures of $\mathscr{m}_\xi$ on the fibers 
 $\set{\xi}\times \C^2$ of the natural projection $\Sigma\times \C^2\to\C^2$ are  atomic $\nu^\Z$-a.e.. But since $\mu$ is invariant we have $\mathscr{m} = \nu^\Z\times \mu$ so $\mathscr{m}_\xi = \mu$ a.s.,  and we are done. 
Thus,  for every $\alpha\in \Delta$
we  must have $\lambda^2_\alpha>0$ or $\lambda^1_\alpha<0$. 
Together with the continuity of the exponents, this  shows that
$P:=\set{\alpha: \ \lambda^2_\alpha>0}$ and $N:=\set{\alpha: \ \lambda^1_\alpha<0}$ form an open cover of 
$\Delta$. But both these open sets must be non-empty: indeed  
 the involution $f\leftrightarrow f\inv$, 
$g\leftrightarrow g\inv$ (i.e. the involution $(a,b,c,d)\mapsto (b,a,d,c)$) changes $(\lambda^1, \lambda^2)$ into $(-\lambda^2, -\lambda^1)$ so it exchanges $P$ and $N$ (here it is important that $\mu$ is $\Gamma$-invariant, so the natural extension of the  product measure on $\Gamma^\N\times \C^2$ is $\nu^\Z\times \mu$, see~\cite[Lem. 13.2]{brh}). So by the connectedness of $\Delta$ we conclude that $P\cap N\neq\emptyset$, which provides the desired hyperbolic parameter.

\begin{rem}\label{rmk:conservative2}
Note that Step 3.1 works exactly the same in the conservative case --it is even simpler since  $\lambda^1_\alpha< \lambda^2_\alpha$ readily implies that $\mu$ is hyperbolic). Also the positivity of the entropy was not used here. 
\end{rem}

\noindent \textit{Step 3.2. Measure classification and conclusion.} 

Roda's measure classification theorem~\cite[Thm 1.1.1]{roda} gives a classification 
of hyperbolic stationary measures  on compact surfaces with $\lambda^1+\lambda^2\geq 0$, 
which applies equally well for 
compactly supported measures in $\C^2$.
Applying the involution $(a,b,c,d)\mapsto (b,a,d,c)$, 
we can guarantee that  $\lambda^1+\lambda^2\geq 0$.
One main point in the proof is to show that unstable directions depend on 
random itineraries, which requires completely different arguments in the $\Aut(\C^2)$ case (see~\cite{nerriere}), but in the context of Case 3 this is granted for free. The upshot is that any such measure is invariant, and it is either atomic, which is not the case here because of the entropy assumption, or its stable and unstable conditionals have an ``extra invariance'' by some group of affine transformations in linearizing coordinates.
We do not need to define exactly what it means at this stage, but this property implies that the Hausdorff 
dimensions of the stable and unstable conditionals coincide, and are equal to 1 or 2. 
By the Ledrappier-Young formula (see~\cite{liu-qian}), this implies that  $\lambda^1+\lambda^2 = 0$. 
On the other hand,  since $\log\abs{\jac f}\neq 0$ and  $P\cap N$ is open, by~\eqref{eq:sum_exponents} 
we can perturb $\alpha$ to make $\lambda^1+\lambda^2 > 0$, which contradicts the above reasoning. 
This completes the proof of Theorem~\ref{thm:same_positive_entropy_measure} in the case where $\mu$ is $f$-ergodic. \qed

Assume now that  $\mu$ is only supposed to be $f$-invariant with $h (f, \mu)>0$
and let us briefly explain how to adapt the proof. We consider $g$ as before, and $\Gamma = \bra{f,g}$, which preserves $\mu$. By the ergodic decomposition theorem for group 
actions~\cite[Thm 4.2]{varadarajan} we can decompose $\mu$ into $\Gamma$-ergodic components as an integral of invariant measures  
$\mu   = \int \mu_{\mathcal P(x)} d\mu(x)$, where $\mathcal P$ is  the measurable partition  into 
$\Gamma$-ergodic components.  
 Since the entropy function is affine, we can find an atom $P\in \mathcal P$ such that 
$\mu_P$ is $\Gamma$-ergodic and $h(f, \mu_P)>0$. Thus, replacing $\mu$ by $\mu_P$
without loss of generality we may assume that  $\mu$ is $\Gamma$-ergodic (but not yet $f$-ergodic).

We review the 3 same cases as before. Assume first that there is a measurable 
$\Gamma$-invariant line field $E$. 
 Let $\mathcal Q$ be the measurable partition  into ergodic components as $f$-invariant measures
 (which is finer than $\mathcal P$) and 
the corresponding decomposition
 $\mu   = \int \mu_{\mathcal Q(x)} d\mu(x)$.  By the affine property of entropy, 
   there exists  a $f$-invariant 
 set $B$ of positive measure such that 
for any $x\in B$, $h_{\mu_{\mathcal Q(x)}}(f)>0$. 
In particular $f$ is non-uniformly hyperbolic along $B$, so  
along $B$, $E$ must coincide almost everywhere with one of $E^s_f$ or $E^u_f$. Reducing $B$ to a smaller 
invariant subset we may  assume that $E = E^s_f$ almost everywhere along $B$. 
Since $g$ preserves $\mu$ we can pick $k\geq 1$ such that 
$\mu(B\cap g^{k}(B))>0$. 
The automorphism $h:=g^{-k} f g^k$ is loxodromic, preserves $\mu$,    
and $g^{-k}(B)$ is a $h$-invariant set along which $h$ is 
is  non-uniformly hyperbolic   and $E$ coincides with $E^s_h$ almost everywhere.
We can now apply the argument of Case 1 to $f$ and $h$ on $B\cap g^{-k}B$, which yields that 
$W^s(f, x) = W^s(h, x)$ for $\mu$-almost every $x\in B\cap g^{-k}B$, hence $T^+_f= T^+_h$. Since
$T^+_h = c (g^k)^*T^+_f$, we get $T^+_f = c(g^k)^*T^+_f $. Since $g$ is loxodromic, this implies 
that $T^+_f = T^+_g$ or  $T^+_f = T^-_g$ so  $\mathrm{Geo}(g)  = \mathrm{Geo}(f)$, possibly up to replacing $g$ by $g\inv$. In Case 2 we argue in the same way. 

For Case 3, since $\mu$ is $\Gamma$-ergodic,  it  is also 
ergodic as a $\nu_\alpha$ stationary measure since this notion depends only on the $\Gamma = \bra{\supp(\nu_\alpha)}$-action. Since $f$ acts hyperbolically on a set of positive measure, 
 we can apply the invariance principle as before  to get that
$\lambda^1_{\alpha}<\lambda^2_\alpha$, and the final argument uses only the ergodicity of $\mu$ as a $\nu_\alpha$-stationary measure so we are done. 
\end{proof}

\section{The conservative case: results and questions}\label{sec:conservative}

\begin{thm}\label{thm:same_eq_measure_conservative}
Let $f$ and $g$ be  polynomial automorphisms of $\C^2$ of positive entropy 
such that $\mu_f = \mu_g$. Assume that $f$ and $g$ are conservative and one of them has a   disconnected Julia set. Then there exists $(M, N)\in (\Z^*)^2$ such that 
$f^N = g^M$. 
\end{thm}

Note that the connectedness of $J_f^*$ is equivalent to that of $J_f$ and $K_f$, and since 
$J_f^* = J_g^*$ the disconnectedness of $J_f$ is equivalent to that of $J_g$. 
No example of a conservative automorphisms with a connected Julia set is known at the moment --in the uniformly hyperbolic case it is a theorem that such mappings do not exist, see~\cite[Cor. A.3]{bs7}-- so that the theorem applies in all known cases.

\begin{proof}
We reproduce the proof of Theorem~\ref{thm:same_eq_measure_dissipative}. As observed in Remarks~\ref{rmk:conservative1} and~\ref{rmk:conservative2}, we can go all the way to Step 3.2 without modification --we only have to deal with the ergodic case here. 
Since for every $\alpha\in \Delta$ we have $\lambda^1_\alpha < \lambda^2_\alpha$ and $\lambda^1_\alpha + \lambda^2_\alpha = 0$ by~\eqref{eq:sum_exponents}, $\mu_f =\mu_g$ is hyperbolic as a stationary measure. Put $\nu = \nu_\alpha$ and $\mu = \mu_f$. 
Since $\mu$ is invariant,  it lifts to a product measure 
$\nu^\Z\times \mu$ in $\Gamma^\Z\times \C^2$. For $\nu^\Z\times \mu$-a.e. $(\xi, x)$ we have Pesin stable and unstable manifolds $W^{s/u}(\xi, x)$, biholomorphic to $\C$.  We can define increasing 
measurable partitions $\eta^{s/u}$ so that for  $\nu^\Z\times \mu$-a.e. $(\xi, x)$ the 
 atom $\eta^{s/u}(\xi, x)$ is an  open subset of $W^{s/u}(\xi, x)$.   
Since $\mu$ is not atomic, the measure classification theorem, 
which is a combination of \cite{roda} and \cite{nerriere},
 asserts  that only two situations can happen:
 \begin{itemize}
 \item either for $\nu^\Z\times \mu$-a.e. $(\xi, x)$, the conditional $\mu(\cdot \vert \eta^{s/u}(\xi,x))$ is absolutely continuous with respect to Lebesgue measure on $W^{s/u}(\xi, x)$;
 \item or for $\nu^\Z\times \mu$-a.e. $(\xi, x)$ there exist analytic curves $U^{s/u}(\xi, x)\subset W^{s/u}(\xi, x)$,  corresponding to real lines in the normalizing coordinates, such that  $\mu(\cdot \vert \eta^{s/u}(\xi,x))$ is absolutely continuous with respect to Lebesgue measure on 
 $U^{s/u}(\xi, x)$.
 \end{itemize}
Here we have identified the conditionals $\mu_\xi$ 
of $\nu^\Z\times \mu$ in $\set{\xi}\times \C^2$ with $\mu$. 
We will refer to these two situations as the SRB and totally real case, respectively. 
To complete the proof of the theorem, we will show that the SRB case does not happen, and that in the totally real case the Julia set must be  connected. 

This will follow from a dimension computation. Indeed, by~\cite[Thm 3.2]{bs3}, the Lyapunov exponents of $\mu_f$ viewed as an invariant measure satisfy $\chi^+(\mu_f) = - \chi^-(\mu_g) \geq \log d$, and by~\cite{bs6}   equality holds if and only if $J_f$ (or equivalently $J_f^*$) is connected. From the Ledrappier-Young  formula~\cite{young} it  follows that 
$\mathrm{HDim}(\mu)\leq 2$ with equality if and only if $J_f$ is connected. 
On the other hand, we can apply the same formula for the conditional measures on $\set{\xi}\times \C^2$ (see e.g.~\cite{liu-xie}) (all of them equal to $\mu$), which implies that $\mathrm{HDim}(\mu)$ equals 4 in the SRB case and 2 in the totally real case. This concludes the proof. 
\end{proof}

\begin{conj}
Theorem~\ref{thm:same_eq_measure_conservative} holds without the disconnectedness assumption. 
\end{conj}

Besides the fact that no conservative automorphisms with connected Julia sets have been spotted so far,   
additional  strong evidence for this conjecture arises from~\cite{rigidity}. First, it is expected that in the 
totally real case, there should  indeed exist 
some totally real invariant surface carrying $\mu$, whose trace 
along the stable and unstable manifolds gives the $U^{s/u}$ curves. This would force the stable and 
unstable slices of $J$ to be contained in smooth curves, a possibility
 that was excluded in~\cite[Thm C]{rigidity}. Establishing the existence of such a surface seems to be a very difficult problem, but the following weaker variant would be sufficient to conclude. The ``joint integrability'' statement  in \cite[Prop 2.6.9]{roda} asserts that on a set of positive measure of fibers 
 $\set{\xi}\times \C^2$, there is a positive measure subset relative to $\mu_\xi$ that is a web of pairwise intersecting analytic curves, parameterized by a Cantor set of positive measure in a line, and  whose topology is compatible with that of lying on a surface. If we could show that this web can be extended to a rectifiable 2-set (even not a dynamically defined one), then~\cite[Thm 3.8]{rigidity} would yield a contradiction.

\section{Further results  in the non-conservative case}\label{sec:further}

We present results towards a complete classification of stationary measures in the following situation: $f$ is dissipative and hyperbolic and $g$ is arbitrary. 
The upshot is that in such a situation non-trivial stationary measures should not exist unless $\bra{f,g}$ is elementary.

If $\nu$ is a finitely supported measure on $\bra{f,g}$, we say that 
$\nu$ is non-dissipative if the average Jacobian $\int \log\abs{\jac(h)} \, \d\nu(h)$ is non-negative. 
Note that in the dissipative case,  
non-trivial stationary measures may exist: indeed if $f$ admits  an attracting fixed point and $\Omega$ is an open set such that $f (\Omega)\Subset \Omega$, we can 
 take $g$ to be a small perturbation of $f$ and  generate an attracting IFS  from $(f,g)$ in $\Omega$. 

\begin{thm}\label{thm:full_classification}
Let $f$, $g$ be loxodromic automorphisms such that the group 
$\bra{f,g}$ is non-elementary and purely loxodromic.  Assume that $f$ is dissipative and hyperbolic.
Let $\nu$ be a non-dissipative, finitely supported measure on $\bra{f,g}$, whose support generates $\bra{f,g}$, and let $\mu$ be a $\nu$-stationary measure. 

Then $\mu$ is atomic, unless the following situation happens: $\mu$ is $\bra{f,g}$-invariant, supported on $J_f$  and of zero entropy for every $h\in \bra{f,g}$, 
$g$ preserves the pair of fields 
$\set{E^s_f, E^u_f}$, but exchanges them on a set of positive measure, and  for every finitely supported 
measure $\nu'$ on $\bra{f,g}$ one the Lyapunov exponents of  $\mu$ (viewed as a $\nu'$-stationary measure)    vanishes.
\end{thm}

This last property is expected to be highly impossible; we will discuss this further in Remark~\ref{rem:unlikely} below. Is seems reasonable to conjecture that under these assumptions, non-atomic stationary measures do not exist (see also Question 1.4 in~\cite{nerriere}). 
We can get an unconditional statement in a perturbative case:
 
\begin{cor}\label{cor:full_classification}
If $f, g$ are as in Theorem~\ref{thm:full_classification} and $g$ is sufficiently close to $f$, then every 
$\nu$-stationary measure is atomic.  In particular if  $X$ is a compact set such that $f(X)= g(X) =X$, then 
$X$ must contain a common periodic orbit of $f$ and $g$. 
\end{cor}

\begin{que}\label{q:is_X_finite}
Let $f$ and $g$ be two polynomial automorphisms such that $\bra{f,g}$ is non-elementary. 
Must  any compact set $X$ such that $f(X)= g(X) =X$ be  finite?  
\end{que}

\begin{proof}[Proof of Theorem~\ref{thm:full_classification}]
In this setting, from \cite[Thm 1.1]{nerriere} we know that $\mu$ is compactly supported and 
the Roda-Nerrière classification of stationary measures~\cite{roda, nerriere}
shows that if $\mu$ is hyperbolic as a stationary measure, then it is atomic, 
hence finitely supported, since there are only finitely many 
$\Gamma$-periodic orbits by \cite{dujardin-favre:DMM, abboud:periodic}. 

We want to study the remaining cases, that is when $\mu$ is not hyperbolic. 
Without loss of generality, we may suppose that $\mu$ is ergodic. 
Since $\nu$ is non-dissipative, the Lyapunov exponents satisfy 
 $  \lambda^2 \geq \lambda^1\geq 0$ so $\mu$ is invariant by the Invariance Principle. So from this point we can freely choose  $\nu$. Let is  pick $\nu = \nu_\alpha$ as in Theorem~\ref{thm:same_eq_measure_dissipative}.  
 Observe  that  since $f$ is hyperbolic and $\mu$ is $f$-invariant,  if 
  $\supp(\mu)$ is not contained in $J_f$, then $\mu$ is automatically atomic, so we may assume 
 $\supp(\mu)\subset J_f$. In particular   
 $E^s_f$ and $E^u_f$ are well-defined and continuous on $\supp(\mu_f)$ 
As in Case 3 of Theorem~\ref{thm:same_positive_entropy_measure} (see Remark~\ref{rmk:conservative2}), if there is no $\Gamma$-invariant line field or pair of invariant line fields, we can find $\alpha$ such that $\mu$ is hyperbolic as a $\nu_\alpha$-stationary measure, hence atomic. If there is a $\Gamma$-invariant line field $V$, then it must coincide with 
  $E^s_f$ or $E^u_f$, say $E^u_f$ for definiteness, 
  and we can calculate the Lyapunov exponent along $V$ by the formula
  $$\chi = \int \lrpar{a\log\norm{Df\rest V}+ b  \log\norm{Df\inv\rest V} +
  c  \log\norm{Dg\rest V} + d  \log\norm{Df\rest V} }\, \d\mu. $$
  In particular, choosing $a$ very close to 1 we get that $\chi$ is very close to 
  $\chi(f, V)=\chi^+(f,\mu)>0$. Then the 
  second Lyapunov  exponent is very close to $\chi^-(f,\mu)<0$, and $\mu$ is hyperbolic again. 
  
So if $\mu$ is not atomic 
 we must be in the situation where $\Gamma$ preserves $\set{E^s_f, E^u_f}$ but $g$ exchanges them 
 on a set of positive measure (but not full measure, otherwise we conclude by replacing $g$ by $g^2$),  
 %In this case we will explain that unless an exceptional property happens, the two  exponents coincide, so they can be chosen to be both positive, and again $\mu$ is atomic. 
and the zero entropy statement follows from Theorem~\ref{thm:same_eq_measure_dissipative}.
The exponents are non-negative, and choosing $\alpha$ such that $\lambda^2+\lambda^1>0$, we must have $\lambda^2>0 = \lambda_1$, because  as already observed, 
positivity of  both exponents would force $\mu$ to be atomic. 
\end{proof}

\begin{rem}\label{rem:unlikely} 
To convince the reader that the remaining situation in Theorem~\ref{thm:full_classification} is truly 
exceptional, let us push its proof a little further. 
%Let $\Sigma_ +  = \set{f,g}^\N$, endowed with the one-sided shift $\sigma$ and the shift invariant measure $\nu^\N$. Then,   
%to the random dynamics  we associate a   skew-product $F$
% on $\Sigma_+\times \C^2$ defined by $F(\omega, x)  =(\sigma(\omega), f_{\omega_0}(x))$. The tangent action of the random dynamics becomes a $\mathrm{GL}_2(\C)$ cocycle over $F$. 
Consider the $\GL_2(\C)$ cocycle over $F_+(\omega, x)  =(\sigma(\omega), f_{\omega_0}(x))$ as in 
\S~\ref{subs:random_GL2}. 
 This cocycle preserves a pair of fields of subspaces $\set{E^s_f, E^u_f}$ by exchanging them on a set $Z$ of positive measure. If $\nu = \nu_\alpha$ and 
 $A$ is the set of points $x$ at which  $Dg_x$ exchanges $E^s_f(x)$ and $E^u_f(x)$, then 
 $Z = [g]\times A \cup [g\inv]\times g(A)$, where we denote by $[h]$ the cylinder of sequences starting with $h$. The ``generic'' case in such a situation is 
 that the    Lyapunov exponents of the cocycle coincide. If so, choosing the weights so that the average Jacobian is positive, we get two positive exponents and again $\mu$ must be atomic.  
So the exponents are distinct, and arguing  as in Knill~\cite{knill} (see also \cite[\S 9.2.2]{viana:book}) shows that $F_Z^2$ is not ergodic, where $F_Z$ denotes the first return map in $Z$ (which is ergodic because $F$ is). The fact that this property holds for all weights $\alpha$ suggests that the return of the $\bra{f,g}$-orbits in $A$ has a quite peculiar structure. Unfortunately we have not been able to exclude this possibility. 
\end{rem}

\begin{proof}[Proof of Corollary~\ref{cor:full_classification}]
If $g$ is close to $f$, then for every $x\in \supp(\mu)$, 
$Dg_x$  must map $E^u_f(x)$ close to $E^u_f(f(x))$, and likewise for 
$E^s$,  a property that  is not compatible with the exchange of directions in Theorem~\ref{thm:full_classification}. So every stationary measure is atomic. The second statement follows from the fact that any $\bra{f,g}$-invariant compact set must support a stationary measure. 
\end{proof}

\begin{rem}\label{rem:Jf=Jg}
Let $f$, $g$ be loxodromic automorphisms such that  $J_f = J_g$. Then 
$\mu_f\asymp \mu_g$ by~\cite{levenberg} or  Theorem~\ref{thm:GZ}.
Assume that $\bra{f,g}$ is non-elementary and 
 fix $N\in \Z$ such that $\bra{f^N, g^N}$ is free. If we can find 
 a finitely supported and non-dissipative probability measure 
 $\nu$ on $\bra{f^N, g^N}$, 
admitting a stationary measure $\mu$ 
in the class of $\mu_f$ and $\mu_g$, then by~\cite{nerriere} $\mu$ is invariant, and by ergodicity we get   
$\mu = \mu_f = \mu_g$. 
\end{rem}

\appendix

\section{Automorphisms and the  harmonic measure}

If $K$ is a non-pluripolar compact  subset in $\C^k$, the pluricomplex Green 
 function $V_K$
is by definition  $$V_K\colon z\mapsto V_K(z) = \sup\set{u(z), \ u\in \el, \ u\rest{K}\leq 0},$$ whose upper semicontinuous regularization $V_K^*$ is psh. Here $\el$ (resp $\el^+$)
is the Lelong class of psh functions in $\C^k$ such that $u(z)\leq \log\norm{z}+ O(1)$ 
(resp. $u(z)= \log\norm{z}+ O(1)$) at infinity. The pluricomplex harmonic  measure\footnote{We prefer this terminology to ``equilibrium measure'' to avoid the confusion with the dynamical one.} 
is by definition 
$\mu_K = (\d\d^c V_K^*)^k$, which is a probability measure (thanks to the internal normalization in $\d\d^c$) supported on $K$.  
The reader is referred to   Klimek's monograph~\cite{klimek} for more details and references. 

For notational simplicity, let us  write $V_K$ for $V_K^*$ 
(note that the equality $V_K = V_K^*$ holds when $K$ is said  regular, which is the case for $K_f$).
If  $\varphi$ is a non-linear  automorphism of $\C^k$, then  
  $(\d\d^c V_K\circ \varphi\inv)^k  = \varphi_* \mu_K$, 
  but $V_K\circ \varphi\inv$ 
is not in $\el$, 
thus $V_K\circ \varphi\inv \neq   V_{\varphi(K)}$. 
This in an indication  that $\varphi_*\mu_K $ should in general 
differ from $ \mu_{\varphi(K)}$, even if  it could  be the case that  equality holds  
for some other reason: for instance if $f$ is a regular automorphism  we have 
$$\mu_{K_f} =  (\d\d^c\max (G^+_f, G^-_f))^2 = f_*\mu_{K_f} = (\d\d^c\max (d\inv G^+_f, dG^-_f))^2.$$
Surprisingly, I could not find any reference where this problem is discussed. 
The  main purpose in this section is to write down an explicit example 
where these two measures are distinct.

Let us start with a positive result. This  is a refinement of~\cite{levenberg},
which   was communicated to me by Vincent Guedj and Ahmed Zeriahi. 

\begin{thm}\label{thm:GZ}
Let $\varphi$ be  an automorphism of $\C^k$,  and put $d =\deg(\varphi)$ and $\delta = \deg(\varphi\inv)$. Then   
 $$\delta^{-k} \varphi_*\mu_K\leq \mu_{\varphi(K)}\leq d^k\varphi_*\mu_K.$$
\end{thm}

The theorem follows rather directly from the following comparison principle.

\begin{lem}[see {\cite[Cor. 3.28]{guedj-zeriahi:book}}]
If $u$ and $v$ are locally bounded psh functions in $\Omega\subset \C^k$ 
such that $u\leq v$ in $\Omega$, then 
$$\mathbf{1}_{\set{u=v}} (\d\d^cu)^k \leq \mathbf{1}_{\set{u=v}} (\d\d^cv)^k$$ as  Borel measures in $\Omega$. 
\end{lem}

\begin{proof}[Proof of the theorem]
Since $V_K\circ \varphi\inv\leq \delta \log\norm{z} + O(1)$ at infinity and $V_K\circ \varphi\inv \leq 0$ on $\varphi(K)$ we get that $\delta\inv V_K\circ \varphi\inv \leq V_{\varphi(K)}$ in $\C^k$.
Hence, since both $(\d\d^c (V_K\circ \varphi\inv))^k$ and $(\d\d^cV_{\varphi(K)})^k$ 
are concentrated on $\set{V_K\circ \varphi = V_{\varphi(K)}=0}$,
 applying the comparison principle  
yields $\delta^{-k} (\d\d^c (V_K\circ \varphi\inv))^k\leq (\d\d^cV_{\varphi(K)})^k$ in $\C^k$, which proves the first inequality. Applying the same argument to $V_{\varphi(K)}\circ \varphi$ and $V_K$ 
yields the other one. 
\end{proof}

We now come to the main example. 

\begin{thm}\label{thm:example_harmonic_measure}
There exists a regular compact set $K$ and an automorphism $\varphi$ such that 
$\varphi_*\mu_K\neq \mu_{\varphi(K)}$.
\end{thm}

\begin{proof}  The point  of this theorem  is that the harmonic measure of a 
symmetric convex body in $\R^k\subset \C^k$ can be computed explicitly (see   Bedford-Taylor~\cite[Thm 1.2]{bedford-taylor:convex}). The construction goes as follows: start with $K = B$ the closed Euclidean unit ball in $\R^2\subset \C^2$. Its harmonic measure is 
$$\mu_B  = \frac{ \d x\,\d y }{2\pi\sqrt{1- x^2 - y^2}}  ,$$ where we denote by $(z,w)$ the coordinates in $\C^2$  $x =\mathrm{Re}(z)$, $y  = \mathrm{Im}(z)$.  
Consider the automorphism $\varphi_\e\colon (z,w)\mapsto (z, w+\e z^3)$ in $\C^2$, where $\e$ is a small positive real number. Then, for small $\e$, $\varphi_\e(B)$ is a symmetric closed convex set in $\R^2$ (this actually holds for $\e\leq 1/6$). Since 
$\varphi_\e$ has Jacobian 1, we have 
$$(\varphi_\e)_*\mu_B = \frac{ \d x\,\d y}{2\pi\sqrt{1- x^2 - (y-\e x^3)^2}}   .$$
Note that value of the  density of $(\varphi_\e)_*\mu_B$ 
at 0 equals $1/2\pi$. 

By Theorem~\ref{thm:GZ}, $\mu_{\varphi_\e(B)}$ is also an absolutely continuous measure supported on $\varphi_\e(B)$, of the form 
$\rho_\e(x,y) \,\d x\, \d y$.

\begin{lem}\label{lem:density_at_0}
With notation as above, $$ \rho_\e (0) = \frac{1}{2\pi} \lrpar{1 - \frac{3}{32}\e^2 + O(\e^3)}.$$
\end{lem}

Thus,   $\rho_\e (0) \neq  1/{2\pi}$ for small $\e$, so
$(\varphi_\e)_*\mu_B\neq \mu_{\varphi_\e(B)}$ and
Theorem~\ref{thm:example_harmonic_measure} follows. \end{proof}

%\begin{cor}
%For small $\e>0$, $(\varphi_\e)_*\mu_B\neq \mu_{\varphi_\e(B)}$. 
%\end{cor}

 \begin{proof}[Proof of the lemma]
 Let $K_\e = \varphi_\e(B)$. The Bedford-Taylor formula~\cite[Thm 1.2]{bedford-taylor:convex} reads 
\begin{equation}\label{eq:BT}
\rho_\e (0) = \frac1{(2\pi)^2}
\int_0^{2\pi}
\frac{d\theta}{h_\varepsilon(\theta)^2} 
%= \frac{1}{2\pi^2}\operatorname{Area}(K_\varepsilon^\circ),
\end{equation}
where $h_\e$ is the support function: $h_\varepsilon(\theta) = \sup_{p\in K} \bra{p, u_\theta}$,   $u_\theta =(\cos \theta, \sin\theta)$. 
 Geometrically, 
 $\rho_\e (0)  =\operatorname{Area}(K_\varepsilon^\circ)/2\pi^2$, where 
  $K_\e^\circ$ is the  polar (or dual) convex body, that is,
\begin{equation*} K_\e^\circ = \set{q\in \R^2\, \vert \, \forall p\in K_\e, \ \bra{p,q}\leq 1}. \end{equation*} 
 The boundary of $K_\varepsilon$ admits the parametrization
\begin{equation*}
 \gamma_\varepsilon(t)
=
\lrpar{\cos t,\,
\sin t+\varepsilon\cos^3t },
\end{equation*}
hence 
\begin{equation}
h_\varepsilon(\theta)
=
\max_{t\in [0, 2\pi]}
\left\langle
\gamma_\varepsilon(t),u_\theta
\right\rangle
=
\max_{t\in [0, 2\pi]}
\left(
\cos(t-\theta)+ \varepsilon\sin\theta\,\cos^3t
\right). \label{eq:max}
\end{equation}
For $\varepsilon=0$, the maximum is attained at $t=\theta$. For $\varepsilon$ small, write
$t=\theta+\delta$ and let $\delta^*$ be the value at which the maximum of $h_\e(\theta+\cdot)$
 is attained, which is close to 0. 
The critical point equation corresponding to~\eqref{eq:max} is
\begin{equation}
-\sin\delta
-
3\varepsilon\sin\theta\,
\cos^2(\theta+\delta)
\sin(\theta+\delta)
=0.
\end{equation}
This first  gives $\delta^* = O(\e)$ and then 
\begin{equation}
\delta^*
=
-3\varepsilon\cos^2\theta\sin^2\theta
+
O(\varepsilon^2).\label{eq:DL}
\end{equation}
Substituting \eqref{eq:DL} into~\eqref{eq:max}, and expanding to the second order then  yields
\begin{equation}
h_\varepsilon(\theta)
=
1
+
\varepsilon\cos^3\theta\sin\theta
+
\frac92\varepsilon^2
\cos^4\theta\sin^4\theta
+
O(\varepsilon^3) 
\end{equation}
so that 
\begin{equation}\label{eq:DL2}
\frac1{h_\varepsilon(\theta)^2}
=
1
-2\varepsilon\cos^3\theta\sin\theta
+
\varepsilon^2
\left(
3\cos^6\theta\sin^2\theta
-
9\cos^4\theta\sin^4\theta
\right)
+
O(\varepsilon^3).
\end{equation}
Using the identities
\begin{equation*}
\int_0^{2\pi}\cos^3\theta\sin\theta\,d\theta=0, \ 
\int_0^{2\pi}
\cos^6\theta\sin^2\theta\,d\theta
=
\frac{5\pi}{64},
\text{ and }
\int_0^{2\pi}
\cos^4\theta\sin^4\theta\,d\theta
=
\frac{3\pi}{64}
\end{equation*}
and inserting these values into~\eqref{eq:DL2} and \eqref{eq:BT}, we obtain
\begin{equation*} \rho_\e(0)=
\frac{1}{4\pi^2}\lrpar{2\pi + \lrpar{3\times \frac{5\pi}{64} - 9 \times \frac{3\pi}{64}}\e^2+ O(\varepsilon^3)}
=
\frac{1}{2\pi}
\left(
1-\frac{3}{32}\varepsilon^2+O(\varepsilon^3)
\right),
\end{equation*} as desired.\end{proof}

\subsection*{Some additional thoughts}
In hindsight, given that a non-linear automorphism of $\C^2$ distorts so much a neighborhood of infinity, it may come as a surprise that invariant harmonic measures even exist! 
This suggests a number of fascinating    questions at the intersection of dynamics and pluripotential theory. 

\begin{que}
If $K$ is a regular compact set in $\C^2$ (or $\C^k$) and $\varphi$ is a loxodromic 
 automorphism   such that  $\varphi_*\mu_K = \mu_K$, must $K$ be ``special''?
 (e.g. a compact set whose Shilov boundary is $J^*$?)
\end{que}

In one-dimensional dynamics, there is an abundant literature on the dynamical approach to the 
geometric properties of harmonic measure. The reason is that when a compact set in $\C$ is invariant under a (typically expanding) conformal dynamical system, there is generally  an invariant measure in the measure class of the harmonic measure (see e.g. \cite{PUZ, zdunik}). 

\begin{que}
 If $K$ is a regular compact set in $\C^k$ invariant under an automorphism $\varphi$, under which conditions does there exist an invariant measure in the class of the pluricomplex harmonic measure? 
\end{que}

Note that in this setting, the measure class of $\mu_K$ is $\varphi_*$-invariant by Theorem~\ref{thm:GZ}. 

 \bibliographystyle{plain}
\bibliography{bib-same}
\end{document}